\documentclass[11pt]{amsart}
\usepackage{amsmath,amsthm,amsfonts,amssymb,txfonts}

\newtheorem{lemma}{Lemma}
\newtheorem{theorem}{Theorem}
\newtheorem{corollary}{Corollary}
\newtheorem{remark}{Remark}
\newtheorem{proposition}{Proposition}
\newtheorem{example}{Example}
\def\bl{\begin{lemma}}
\def\bt{\begin{theorem}}
\def\el{\end{lemma}}
\def\et{\end{theorem}}
\def\bp{\begin{proof}}
\def\ep{\end{proof}}
\def\bc{\begin{corollary}}
\def\ec{\end{corollary}}
\def\bc{\begin{remark}}
\def\ec{\end{remark}}
\subjclass[2000]{32A10, 32A36, 51M25} 

\begin{document}

\title[Weighted Integral Means of Mixed Areas and Lengths]{Weighted Integral Means of Mixed\\
 Areas and Lengths under Holomorphic Mappings}
\author{Jie Xiao and Wen Xu}
\address{Department of Mathematics and Statistics,
         Memorial University,
         NL A1C 5S7, Canada}
         \email{jxiao@mun.ca; wenxupine@gmail.com}
\thanks{JX and WX were in part supported by NSERC of Canada and the Finnish Cultural Foundation, respectively.}

\begin{abstract}
This note addresses monotonic growths and logarithmic convexities of
the weighted ($(1-t^2)^\alpha dt^2$, $-\infty<\alpha<\infty$,
$0<t<1$) integral means $\mathsf{A}_{\alpha,\beta}(f,\cdot)$ and
$\mathsf{L}_{\alpha,\beta}(f,\cdot)$ of the mixed area $(\pi
r^2)^{-\beta}A(f,r)$ and the mixed length $(2\pi r)^{-\beta}L(f,r)$
($0\le\beta\le 1$ and $0<r<1$) of $f(r\mathbb D)$ and $\partial
f(r\mathbb D)$ under a holomorphic map $f$ from the unit disk
$\mathbb D$ into the finite complex plane $\mathbb C$.
\end{abstract}
 \maketitle
\section{Introduction}
From now on, $\mathbb D$ represents the unit disk in the finite
complex plane $\mathbb C$, $H(\mathbb D)$ denotes the space of
holomorphic mappings $f: \mathbb D\to\mathbb C$, and $U(\mathbb D)$
stands for all univalent functions in $H(\mathbb D)$. For any real
number $\alpha$, positive number $r\in (0,1)$ and the standard area
measure $dA$, let
$$
dA_\alpha(z)=(1-|z|^2)^\alpha dA(z);\quad r\mathbb D=\{z\in\mathbb
D: |z|<r\};\quad r\mathbb T=\{z\in\mathbb D: |z|=r\}.
$$

In their recent paper \cite{XZ}, Xiao and Zhu have discussed the
following area $0<p<\infty$-integral means of $f\in H(\mathbb D)$:
$$
{M}_{p,\alpha}(f,r)=\left[\frac{1}{A_\alpha(r\mathbb
D)}\int_{r\mathbb D}|f|^p\,dA_\alpha\right]^{\frac1p},
$$
proving that $r\mapsto M_{p,\alpha}(f,r)$ is strictly increasing
unless $f$ is a constant, and $\log r\mapsto\log M_{p,\alpha}(f,r)$
is not always convex. This last result suggests a conjecture that
$\log r\mapsto\log M_{p,\alpha}(f,r)$ is convex or concave when
$\alpha\le 0$ or $\alpha>0$. But, motivated by \cite[Example 10,
(ii)]{XZ} we can choose $p=2$, $\alpha=1$, $f(z)=z+c$ and $c>0$ to
verify that the conjecture is not true. At the same time, this
negative result was also obtained in Wang-Zhu's manuscript
\cite{WZ}. So far it is unknown whether the conjecture is generally
true for $p\not=2$.

The foregoing observation has actually inspired the following
investigation. Our concentration is the fundamental case $p=1$. To
understand this approach, let us take a look at
$M_{1,\alpha}(\cdot,\cdot)$ from a differential geometric viewpoint.
Note that
$$
{M}_{1,\alpha}(f',r)=\frac{\int_{r\mathbb
D}|f'|\,dA_\alpha}{A_\alpha(r\mathbb D)}=\frac{\int_0^r \big[(2\pi
t)^{-1}\int_{t\mathbb
T}|f'(z)||dz|\big](1-t^2)^\alpha\,dt^2}{\int_0^r
(1-t^2)^\alpha\,dt^2}.
$$
So, if $f\in U(\mathbb D)$, then
$$
(2\pi t)^{-1}\int_{t\mathbb T}|f'(z)|\,|dz|
$$
is a kind of mean of the length of $\partial f(t\mathbb D)$, and
hence the square of this mean dominates a sort of mean of the area
of $f(t\mathbb D)$ in the isoperimetric sense:
$$
\Phi_{A}(f,t)=(\pi t^2)^{-1}\int_{t\mathbb D}|f'(z)|^2\,dA(z)\le
\left[(2\pi t)^{-1}\int_{t\mathbb
T}|f'(z)|\,|dz|\right]^2=\big[\Phi_{L}(f,t)\big]^2.
$$
According to the P\'olya-Szeg\"o monotone principle \cite[Problem
309]{PS} (or \cite[Proposition 6.1]{BMM}) and the area Schwarz's
lemma in Burckel, Marshall, Minda, Poggi-Corradini and Ransford
\cite[Theorem 1.9]{BMM}, $\Phi_{L}(f,\cdot)$ and $\Phi_{A}(f,\cdot)$
are strictly increasing on $(0,1)$ unless $f(z)=a_1z$ with
$a_1\not=0$. Furthermore, $\log\Phi_{L}(f,r)$ and
$\log\Phi_{A}(f,r)$, equivalently, $\log L(f,r)$ and $\log A(f,r)$,
are convex functions of $\log r$ for $r\in (0,1)$, due to the
classical Hardy's convexity and \cite[Section 5]{BMM}. Perhaps, it
is worth-wise to mention that if $c>0$ is small enough then the
universal cover of $\mathbb D$ onto the annulus
$\{e^{-\frac{c\pi}{2}}<|z|< e^{\frac{c\pi}{2}}\}$:
$$
f(z)=\exp\Big[ic\log\Big(\frac{1+z}{1-z}\Big)\Big]
$$
enjoys the property that $\log r\mapsto \log A(f,r)$ is not convex;
see \cite[Example 5.1]{BMM}.

In the above and below, we have used the following convention:
$$
\Phi_{A}(f,r)=\frac{A(f,r)}{\pi r^2}\quad\&\quad
\Phi_{L}(f,r)=\frac{L(f,r)}{2\pi r},
$$
where under $r\in (0,1)$ and $f\in H(\mathbb D)$, $A(f,r)$ and
$L(f,r)$ stand respectively for the area of $f(r\mathbb D)$ (the
projection of the Riemannian image of $r\mathbb D$ by $f$) and the
length of $\partial f(r\mathbb D)$ (the boundary of the projection
of the Riemannian image of $r\mathbb D$ by $f$) with respect to the
standard Euclidean metric on $\mathbb C$. For our purpose, we choose
a shortcut notation
$$
d\mu_\alpha(t)=(1-t^2)^\alpha dt^2\quad\&\quad
\nu_\alpha(t)=\mu_\alpha([0,t])\quad\forall\quad t\in (0,1),
$$
and for $0\le\beta\le 1$ define
$$
\Phi_{A,\beta}(f,t)=\frac{A(f,t)}{(\pi t^2)^\beta}\quad\&\quad
\Phi_{L,\beta}(f,t)=\frac{L(f,t)}{(2\pi t)^\beta},
$$
and then
$$
\mathsf{A}_{\alpha,\beta}(f,r)=\frac{\int_0^r \Phi_{A,\beta}(f,t)
\,d\mu_\alpha(t)}{\int_0^r d\mu_\alpha(t)}\quad\&\quad
\mathsf{L}_{\alpha,\beta}(f,r)=\frac{\int_0^r \Phi_{L,\beta}(f,t)\,
d\mu_\alpha(t)}{\int_0^r d\mu_\alpha(t)}
$$
which are called the weighted integral means of the mixed area and
the mixed length for $f(r\mathbb D)$ and $\partial f(r\mathbb D)$,
respectively.

In this note, we consider two fundamental properties: monotonic
growths and logarithmic convexities of both
$\mathsf{A}_{\alpha,\beta}(f,r)$ and
$\mathsf{L}_{\alpha,\beta}(f,r)$, thereby producing two
specialities: (i) if $r\mapsto \Phi_{L}(f,r)$ is monotone increasing
on $(0,1)$, then so is the isoperimetry-induced function:
$$
r\mapsto\frac{\int_0^r
\big[\Phi_{L,1}(f,t)\big]^2\,d\mu_\alpha(t)}{\int_0^r
d\mu_\alpha(t)}\ge \mathsf{A}_{\alpha,1}(f,r);
$$
(ii) the log-convexity for $\mathsf{L}_{\alpha,\beta=1}(f,r)$
essentially settles the above-mentioned conjecture. The details
(results and their proofs) are arranged in the forthcoming two
sections.

\section{Monotonic Growth}

In this section, we deal with the monotonic growths of
$\mathsf{A}_{\alpha,\beta}(f,r)$ and
$\mathsf{L}_{\alpha,\beta}(f,r)$, along with their associated
Schwarz type lemmas. In what follows, $\mathbb N$ is used as the set
of all natural numbers.

\subsection{Two Lemmas} The following two preliminary results are
needed.

\begin{lemma}\cite[Theorems 1 \& 2]{Ma}\label{l2} Let $f\in H(\mathbb D)$ be of the form
$f(z)=a_0+\sum_{k=n}^\infty a_kz^k$ with $n\in\mathbb N$. Then:

\item{\rm(i)} $\pi r^{2n}\Big[\frac{|f^{(n)}(0)|}{n!}\Big]^2\le A(f,r)\quad\forall\quad r\in
(0,1)$.

\item{\rm(ii)} $2\pi r^n \Big[\frac{|f^{(n)}(0)|}{n!}\Big]\le L(f,r)\quad\forall\quad r\in
(0,1)$.

\noindent Moreover, equality in (i) or (ii) holds if and only if
$f(z)=a_0+a_nz^n$.
\end{lemma}

\begin{proof} This may be viewed as the higher order Schwarz type lemma for area and length.
See also the proofs of Theorems 1 \& 2 in \cite{Ma}, and their
immediate remarks on equalities. Here it is worth noticing three
matters: (a) $\frac{f^{(n)}(0)}{n!}$ is just $a_n$; (b)
\cite[Corollary 3]{J} presents a different argument for the area
case; (c) $L(f,r)$ is greater than or equal to the length $l(r,f)$
of the outer boundary of $f(r\mathbb D)$ (defined in \cite{Ma})
which is not less than the length $l^\#(r,f)$ of the exact outer
boundary of $f(r\mathbb D)$ (introduced in \cite{Y}).
\end{proof}

\begin{lemma}\label{l1} Let $0\le\beta\le 1$.

\item{\rm(i)} If $f\in H(\mathbb D)$, then $r\mapsto \Phi_{A,\beta}(f,r)$ is strictly
increasing on $(0,1)$ unless
\[
f=\left\{\begin{array} {r@{\;}l}
constant &\quad \hbox{when}\quad \beta<1\\
linear\ map &\quad \hbox{when}\quad \beta=1.
\end{array}
\right.
\]

\item{\rm(ii)} If $f\in U(\mathbb D)$ or $f(z)=a_0+a_nz^n$ with $n\in\mathbb N$, then $r\mapsto
\Phi_{L,\beta}(f,r)$ is strictly increasing on $(0,1)$ unless
\[
f=\left\{\begin{array} {r@{\;}l}
constant &\quad \hbox{when}\quad \beta<1\\
linear\ map & \quad \hbox{when}\quad \beta=1.
\end{array}
\right.
\]

\end{lemma}
\begin{proof} It is enough to handle $\beta<1$ since the case $\beta=1$ has been treated in \cite[Theorem 1.9 \& Proposition 6.1]{BMM}.
The monotonic growths in (i) and (ii) follow from
$$
\Phi_{A,\beta}(f,r)=(\pi r^2)^{1-\beta}\Phi_{A,1}(f,r)\quad\&\quad
L(f,r)=(2\pi r)^{1-\beta}\Phi_{L,1}(f,r).
$$

To see the strictness, we consider two cases.

(i) Suppose that $\Phi_{A,\beta}(f,\cdot)$ is not strictly
increasing. Then there are $r_1,r_2\in (0,1)$ such that $r_1<r_2$,
and $\Phi_{A,\beta}(f,\cdot)$ is a constant on $[r_1,r_2]$. Hence
$$
\frac{d}{dr}\Phi_{A,\beta}(f,r)=0\quad\forall\quad r\in [r_1,r_2].
$$
Equivalently,
$$
2\beta A(f,r)=r\frac{d}{dr}A(f,r)\quad\forall\quad r\in [r_1,r_2].
$$
But, according to \cite[(4.2)]{BMM}:
$$
2A(f,r)\le r\frac{d}{dr} A(f,r)\quad\forall\quad r\in (0,1).
$$
Since $\beta<1$, we get $A(f,r)=0$ for all $r\in [r_1,r_2]$, whence
finding that $f$ is constant.

(ii) Now assume that $\Phi_{L,\beta}(f,\cdot)$ is not strictly
increasing. There are $r_3,r_4\in (0,1)$ such that $ r_3<r_4$ and
$$
0=\frac{d}{dr}\Phi_{L,\beta}(f,r)=(2\pi
r)^{-\beta}\Big[\frac{d}{dr}L(f,r)-\frac{\beta}{r}L(f,r)\Big]=0\quad\forall\quad
r\in [r_3,r_4].
$$
If $f\in U(\mathbb D)$ then
$$
L(f,r)=\int_{r\mathbb T}|f'(z)|\,|dz|
$$
and hence one has the following ``first variation formula"
$$
\frac{d}{dr}L(f,r)=\int_0^{2\pi}|f'(re^{i\theta})|d\theta+r\frac{d}{dr}\int_0^{2\pi}|f'(re^{i\theta})|d\theta\quad\forall\quad
r\in [r_3,r_4].
$$
The previous three equations yield
$$
0=(1-\beta)\int_0^{2\pi}|f'(re^{i\theta})|d\theta+r\frac{d}{dr}\int_0^{2\pi}|f'(re^{i\theta})|d\theta\quad\forall\quad
r\in [r_3,r_4]
$$
and so
$$
\int_0^{2\pi}|f'(re^{i\theta})|d\theta=0\quad\forall\quad r\in
[r_3,r_4].
$$
This ensures that $f$ is a constant, contradicting $f\in U(\mathbb
D)$. Therefore, $f(z)$ is of the form $a_0+a_nz^n$. But, since
$L(z^n,r)=2\pi r^n$ is strictly increasing, $f$ must be constant.
\end{proof}

\subsection{Monotonic Growth of $\mathsf{A}_{\alpha,\beta}(f,\cdot)$} This
aspect is essentially motivated by the following Schwarz type lemma.

\begin{proposition}\label{pr1} Let $-\infty<\alpha<\infty$, $0\le\beta\le 1$, and $f\in H(\mathbb D)$
be of the form $f(z)=a_0+\sum_{k=n}^\infty a_k z^k$ with $n\in\mathbb N$.
Then
$$
\pi^{1-\beta}\Big[\frac{|f^{(n)}(0)|}{n!}\Big]^2\le
\mathsf{A}_{\alpha,\beta}(f,r)\left[\frac{\nu_\alpha(r)}{\int_0^rt^{2(n-\beta)}\,d\mu_\alpha(t)}\right]\quad\forall\quad
r\in (0,1)
$$
with equality if and only if $f(z)=a_0+a_nz^n$.
\end{proposition}

\begin{proof} The inequality follows from Lemma \ref{l2} (i) right away. When $f(z)=a_0+a_nz^n$,
the last inequality becomes equality due to the equality case of
Lemma \ref{l2} (i). Conversely, suppose that the last inequality is
an equality. If $f$ does not have the form $a_0+a_nz^n$, then the
equality in Lemma \ref{l2} (i) is not true, then there are
$r_1,r_2\in (0,1)$ such that $r_1<r_2$ and
$$
A(f,t)>\pi
t^{2n}\Big[\frac{|f^{(n)}(0)|}{n!}\Big]^2\quad\forall\quad t\in
[r_1,r_2].
$$
This strict inequality forces that for $r\in [r_1,r_2]$,
\begin{eqnarray*}
\pi^{1-\beta}\Big[\frac{|f^{(n)}(0)|}{n!}\Big]^2\int_0^r
t^{2(n-\beta)}\,d\mu_\alpha(t)&=&\int_0^r
(\pi t^2)^{-\beta}A(f,t)\,d\mu_\alpha(t)\\
&=&\left(\int_0^{r_1}+\int_{r_1}^{r_2}+\int_{r_2}^{r}\right)(\pi t^2)^{-\beta} A(f,t)\,d\mu_\alpha(t)\\
&>&\pi^{1-\beta} \Big[\frac{|f^{(n)}(0)|}{n!}\Big]^2 \int_0^{r}
t^{2(n-\beta)}\,d\mu_\alpha(t),
\end{eqnarray*}
a contradiction. Thus $f(z)=a_0+a_nz^n$.
\end{proof}

Based on Proposition \ref{pr1}, we find the monotonic growth for
$\mathsf{A}_{\alpha,\beta}(\cdot,\cdot)$ as follows.

\begin{theorem}\label{th1}
Let $-\infty<\alpha<\infty$, $0\le\beta\le 1$, and $f\in H(\mathbb
D)$. Then $r\mapsto\mathsf{A}_{\alpha,\beta}(f,r)$ is strictly
increasing on $(0,1)$ unless
\[
f=\left\{\begin{array} {r@{\;}l}
constant &\quad \hbox{when}\quad \beta<1\\
linear\ map &\quad \hbox{when}\quad \beta=1.
\end{array}
\right.
\]
Consequently,

\item{\rm(i)}
\[
\lim_{r\to 0}\mathsf{A}_{\alpha,\beta}(f,r)=\left\{\begin{array}
{r@{\;}l}
0\quad & \hbox{when}\quad \beta<1\\
|f'(0)|^2\quad & \hbox{when}\quad \beta=1.
\end{array}
\right.
\]

\item{\rm(ii)} If
$$
\Phi_{A,\beta}(f,0):=\lim_{r\to
0}\Phi_{A,\beta}(f,r)\quad\&\quad\Phi_{A,\beta}(f,1):=\lim_{r\to
1}\Phi_{A,\beta}(f,r)<\infty,
$$
then
$$
0<r<s<1\Rightarrow 0\le
\frac{\mathsf{A}_{\alpha,\beta}(f,s)-\mathsf{A}_{\alpha,\beta}(f,r)}{\log\nu_\alpha(s)-\log\nu_\alpha(r)}\leq
\Phi_{A,\beta}(f,s)-\Phi_{A,\beta}(f,0)
$$
with equality if and only if
\[
f=\left\{\begin{array} {r@{\;}l}
\hbox{constant}\quad & \hbox{when}\quad \beta<1\\
\hbox{linear\ map}\quad & \hbox{when}\quad \beta=1.
\end{array}
\right.
\]
In particular, $t\mapsto \mathsf{A}_{\alpha,\beta}(f,t)$ is
Lipschitz with respect to $\log\nu_\alpha(t)$ for $t\in (0,1)$.

\end{theorem}
\begin{proof} Note that $\nu_\alpha(r)=\int_0^r d\mu_\alpha(t)$. So $d\nu_\alpha(r)$, the differential of $\nu_\alpha(r)$ with respect to $r\in (0,1)$, equals
$d\mu_\alpha(r)$.  By integration by parts we have
$$
\Phi_{A,\beta}(f,r)\nu_\alpha(r)-\int_0^r
\Phi_{A,\beta}(f,t)\,d\mu_\alpha(t)=\int_0^r\big[\frac{d}{dt}\Phi_{A,\beta}(f,t)\big]
\nu_\alpha(t)\,dt.
$$
Differentiating the function $\mathsf{A}_{\alpha,\beta}(f,r)$ with
respect to $r$ and using Lemma \ref{l1} (i), we get
\begin{align*}
\frac{d}{dr}\mathsf{A}_{\alpha,\beta}(f,r)&=\frac{\Phi_{A,\beta}(f,r)2r(1-r^2)^\alpha \nu_\alpha(r)-\Big[\int_0^r\Phi_{A,\beta}(f,t)\, d\mu_\alpha(t)\Big]2r(1-r^2)^\alpha}{\nu_\alpha(r)^2}\\
&=\frac{2r(1-r^2)^\alpha \left[\Phi_{A,\beta}(f,t)\nu_\alpha(r)- \int_0^r \Phi_{A,\beta}(f,t)\, d\mu_\alpha(t)\right]}{\nu_\alpha(r)^2}\\
&=\frac{2r(1-r^2)^\alpha\int_0^r\big[\frac{d}{dt}\Phi_{A,\beta}(f,t)\big]
\nu_\alpha(t)\, dt}{\nu_\alpha(r)^2}\geq 0.
\end{align*}
As a result,
$r\mapsto\mathsf{A}_{\alpha,\beta}(f,r)$ increases on $(0,1)$.

Next suppose that the just-verified monotonicity is not strict. Then
there exist two numbers $r_1,r_2\in (0,1)$ such that $r_1<r_2$ and
$$
\mathsf{A}_{\alpha,\beta}(f,r_1)=\mathsf{A}_{\alpha,\beta}(f,r)=\mathsf{A}_{\alpha,\beta}(f,r_2)\quad
\forall\quad r\in [r_1,r_2].
$$
Consequently,
$$
\frac{d}{dr}\mathsf{A}_{\alpha,\beta}(f,r)=0\quad\forall\quad
r\in[r_1,r_2]
$$
and so
$$
\int_0^r \big[\frac{d}{dt}\Phi_{A,\beta}(f,t)\big]\nu_\alpha(t)\,
dt=0\quad\forall\quad r\in [r_1,r_2].
$$
Then we must have
$$
\frac{d}{dt}\Phi_{A,\beta}(f,t)=0\quad\forall\quad t\in
(0,r)\quad\hbox{with}\quad r\in [r_1,r_2],
$$
whence getting that if $\beta<1$ then $f$ must be constant or if
$\beta=1$ then $f$ must be linear, thanks to the argument for the
strictness in Lemma \ref{l1} (i).

It remains to check the rest of Theorem \ref{th1}.

(i) The monotonic growth of $\mathsf{A}_{\alpha,\beta}(f,\cdot)$
ensures the existence of the limit. An application of
L'H\^{o}pital's rule gives
$$
\lim_{r\to 0}\mathsf{A}_{\alpha,\beta}(f,r)=\lim_{r\to
0}\Phi_{A,\beta}(f,r)= \left\{\begin{array} {r@{\;}l}
0\quad & \hbox{when}\quad \beta<1\\
|f'(0)|^2\quad & \hbox{when}\quad \beta=1.
\end{array}
\right.
$$

(ii) Again, the above monotonicity formula of
$\mathsf{A}_{\alpha,\beta}(f,\cdot)$ plus the given condition yields
that for $s\in (0,1)$,
$$
\sup_{r\in
(0,s)}\mathsf{A}_{\alpha,\beta}(f,r)=\mathsf{A}_{\alpha,\beta}(f,s)<\infty.
$$
Integrating by parts twice and using the monotonicity of
$\Phi_{A,\beta}(f,\cdot)$, we obtain that under $0<r<s<1$,
\begin{eqnarray*}
0&\le&\mathsf{A}_{\alpha,\beta}(f,s)-\mathsf{A}_{\alpha,\beta}(f,r)\\
&=&\int_r^s\frac{d}{dt}\mathsf{A}_{\alpha,\beta}(f,t)\,dt\\
&=&\int_r^s\left(\int_0^t\big[\frac{d}{d\tau}\Phi_{A,\beta}(f,\tau)\big]\nu_\alpha(\tau)\,d\tau\right)\,\Big[\frac{d\nu_\alpha(t)}{\nu_\alpha(t)^2}\Big]\\
&=&\int_r^s\left(\nu_\alpha(t)\Phi_{A,\beta}(f,t)-\int_0^t\Phi_{A,\beta}(f,\tau)\,d\nu_\alpha(\tau)\right)\,\Big[\frac{d\nu_\alpha(t)}{\nu_\alpha(t)^2}\Big]\\
&\le&\Big[\Phi_{A,\beta}(f,s)-\Phi_{A,\beta}(f,0)\Big]\int_r^s\frac{d\nu_\alpha(t)}{\nu_\alpha(t)}.
\end{eqnarray*}
This gives the desired inequality right away. Furthermore, the above
argument plus Lemma \ref{l1} (i) derives the equality case.

\end{proof}

As an immediate consequence of Theorem \ref{th1}, we get a sort of
``norm" estimate associated with $\Phi_{A,\beta}(f,\cdot)$.

\begin{corollary}\label{pr2} Let $-\infty<\alpha<\infty$, $0\le\beta\le 1$, and $f\in H(\mathbb D)$.
\item{\rm(i)} If $-\infty<\alpha\le -1$, then
$$
\int_0^1 \Phi_{A,\beta}(f,t)\,d\mu_\alpha(t)=\sup_{r\in
(0,1)}\int_0^r \Phi_{A,\beta}(f,t)\,d\mu_\alpha(t)<\infty
$$
if and only if $f$ is constant. Moreover, $\sup_{r\in
(0,1)}\mathsf{A}_{\alpha,\beta}(f,r)=\Phi_{A,\beta}(f,1).$

\item{\rm(ii)} If $-1<\alpha<\infty$, then
$$
\mathsf{A}_{\alpha,\beta}(f,r)\le\mathsf{A}_{\alpha,\beta}(f,1):=\sup_{s\in
(0,1)}\mathsf{A}_{\alpha,\beta}(f,s)\quad\forall\quad r\in (0,1),
$$
where the inequality becomes an equality for all $r\in (0,1)$ if and
only if
\[
f=\left\{\begin{array} {r@{\;}l}
\hbox{constant}\quad & \hbox{when}\quad \beta<1\\
\hbox{linear\ map}\quad & \hbox{when}\quad \beta=1.
\end{array}
\right.
\]

\item{\rm(iii)} The following function $\alpha\mapsto\mathsf{A}_{\alpha,\beta}(f,1)$ is strictly decreasing on
$(-1,\infty)$ unless
\[
f=\left\{\begin{array} {r@{\;}l}
\hbox{constant}\quad & \hbox{when}\quad \beta<1\\
\hbox{linear\ map}\quad & \hbox{when}\quad \beta=1.
\end{array}
\right.
\]
\end{corollary}

\begin{proof} (i) By Theorem \ref{th1}, we have
$$
\mathsf{A}_{\alpha,\beta}(f,r)\leq \frac{\int_0^s
\Phi_{A,\beta}(f,t)\,d\mu_\alpha(t)}{\nu_\alpha(s)}\quad\forall\quad
r\in (0,s).
$$
Note that
$$\lim_{s\to 1}\nu_\alpha(s)=\infty\quad\&\quad\lim_{s\to 1}\int_0^s\Phi_{A,\beta}(f,t)\, d\mu_\alpha(t)=\int_0^1
\Phi_{A,\beta}(f,t)\,d\mu_\alpha(t).
$$
So, the last integral is finite if and only if
$$
\Phi_{A,\beta}(f,r)=0\quad\forall\quad r\in (0,1),
$$
equivalently, $A(f,r)=0$ holds for all $r\in (0,1)$, i.e., $f$ is
constant.

For the remaining part of (i), we may assume that $f$ is not a
constant map. Due to $\lim_{r\to 1}\nu_\alpha(r)=\infty$, we obtain
$$
\lim_{r\to 1}\int_0^r \Phi_{A,\beta}(f,t)\,d\mu_\alpha(t)=\int_0^1
\Phi_{A,\beta}(f,t)\,d\mu_\alpha(t)=\infty.
$$
So, an application of L'H\^{o}pital's rule yields
$$
\sup_{0<r<1}\mathsf{A}_{\alpha,\beta}(f,r)=\lim_{r\to
1}\frac{\int_0^r \Phi_{A,\beta}(f,t)\,
d\mu_\alpha(t)}{\nu_\alpha(r)}=\lim_{r\to
1}\frac{\Phi_{A,\beta}(f,r)r(1-r^2)^\alpha}{
r(1-r^2)^\alpha}=\Phi_{A,\beta}(f,1).
$$

(ii) Under $-1<\alpha<\infty$, we have
$$
\lim_{r\to 1}\nu_\alpha(r)=\nu_\alpha(1)\quad\&\quad \lim_{r\to
1}\int_0^r\Phi_{A,\beta}(f,t)\,d\mu_\alpha(t)=\int_0^1
\Phi_{A,\beta}(f,t)\,d\mu_\alpha(t).
$$
Thus, by Theorem \ref{th1} it follows that for $r\in (0,1)$,
$$
\mathsf{A}_{\alpha,\beta}(f,r)\le\lim_{s\to
1}\mathsf{A}_{\alpha,\beta}(f,s)=\big[\nu_\alpha(1)\big]^{-1}\int_0^1
\Phi_{A,\beta}(f,t)\, d\mu_\alpha(t)=\sup_{s\in
(0,1)}\mathsf{A}_{\alpha,\beta}(f,s).
$$
The equality case just follows from a straightforward computation
and Theorem \ref{th1}.

(iii) Suppose $-1<\alpha_1<\alpha_2<\infty$ and
$\mathsf{A}_{\alpha_1,\beta}(f,1)<\infty$, then integrating by parts
twice, we obtain

\begin{align*}
\mathsf{A}_{\alpha_2,\beta}(f,1)&=
\big[\nu_{\alpha_2}(1)\big]^{-1}\int_0^1\Phi_{
A,\beta}(f,r)\,d\mu_{\alpha_2}(r)\\
&= \big[\nu_{\alpha_2}(1)\big]^{-1}\int_0^1 (1-r^2)^{\alpha_2-\alpha_1}\frac{d}{dr}\left[\int_0^r \Phi_{A,\beta}(f,t)\, d\mu_{\alpha_1}(t)\right]\, dr\\
&= \big[\nu_{\alpha_2}(1)\big]^{-1}\left[-\int_0^1\left(\int_0^r\Phi_{A,\beta}(f,t)\,d\mu_{\alpha_1}(t)\right)\, d(1-r^2)^{\alpha_2-\alpha_1}\right]\\
&\leq \big[\nu_{\alpha_2}(1)\big]^{-1}\mathsf{A}_{\alpha_1,\beta}(f,1)\int_0^1 \nu_{\alpha_1}(r)\, d\big[-(1-r^2)^{\alpha_2-\alpha_1}\big]\\
&=\mathsf{A}_{\alpha_1,\beta}(f,1)\big[\nu_{\alpha_2}(1)\big]^{-1}\left[\int_0^1
(1-r^2)^{\alpha_2-\alpha_1}\,d\mu_{\alpha_1}(r)\right]
\\
&=\mathsf{A}_{\alpha_1,\beta}(f,1),
\end{align*}
thereby establishing $\mathsf{A}_{\alpha_2,\beta}(f,1)\le
\mathsf{A}_{\alpha_1,\beta}(f,1)$. If this last inequality becomes
equality, then the above argument forces
$$
\int_0^r\Phi_{A,\beta}(f,t)\,d\mu_{\alpha_1}(t)=\mathsf{A}_{\alpha_1,\beta}(f,1)
\nu_{\alpha_1}(r)\quad\forall\quad r\in (0,1),
$$
whence yielding (via the just-verified (ii))
\[ f=\left\{\begin{array} {r@{\;}l}
\hbox{constant}\quad & \hbox{when}\quad \beta<1\\
\hbox{linear\ map}\quad & \hbox{when}\quad \beta=1.
\end{array}
\right.
\]
\end{proof}

\subsection{Monotonic Growth of $\mathsf{L}_{\alpha,\beta}(f,\cdot)$} Correspondingly, we first have
the following Schwarz type lemma.

\begin{proposition}\label{co1} Let $-\infty<\alpha<\infty$, $0\le\beta\le 1$, and $f\in H(\mathbb D)$ be of the form $f(z)=a_0+\sum_{k=n}^\infty a_kz^k$ with $n\in\mathbb N$.
Then
$$
(2\pi)^{1-\beta}\Big[\frac{|f^{(n)}(0)|}{n!}\Big]\le
\mathsf{L}_{\alpha,\beta}(f,r)\left[\frac{\nu_\alpha(r)}{\int_0^rt^{n-\beta}\,d\mu_\alpha(t)}\right]\quad\forall\quad
r\in (0,1)
$$
with equality when and only when $f=a_0+a_nz^n$.
\end{proposition}

\begin{proof} This follows from Lemma \ref{l2} (ii) and its equality
case.
\end{proof}

The coming-up-next monotonicity contains a hypothesis stronger than
that for Theorem \ref{th1}.

\begin{theorem}\label{th2}
Let $-\infty<\alpha<\infty$, $0\le\beta\le 1$, and $f\in U(\mathbb
D)$ or $f(z)=a_0+a_nz^n$ with $n\in\mathbb N$. Then
$r\mapsto\mathsf{L}_{\alpha,\beta}(f,r)$ is strictly increasing on
$(0,1)$ unless
\[
f=\left\{\begin{array} {r@{\;}l}
constant &\quad \hbox{when}\quad \beta<1\\
linear\ map &\quad \hbox{when}\quad \beta=1.
\end{array}
\right.
\]
Consequently,

\item{\rm(i)}
\[
\lim_{r\to 0}\mathsf{L}_{\alpha,\beta}(f,r)=\left\{\begin{array}
{r@{\;}l}
0\quad & \hbox{when}\quad \beta<1\\
|f'(0)|\quad & \hbox{when}\quad \beta=1.
\end{array}
\right.
\]

\item{\rm(ii)} If
$$
\Phi_{L,\beta}(f,0):=\lim_{r\to
0}\Phi_{L,\beta}(f,r)\quad\&\quad\Phi_{L,\beta}(f,1):=\lim_{r\to
1}\Phi_{L,\beta}(f,r)<\infty,
$$
then
$$
0<r<s<1\Rightarrow 0\le
\frac{\mathsf{L}_{\alpha,\beta}(f,s)-\mathsf{L}_{\alpha,\beta}(f,r)}{\log\nu_\alpha(s)-\log\nu_\alpha(r)}\leq
\Phi_{L,\beta}(f,s)-\Phi_{L,\beta}(f,0)
$$
with equality if and only if
\[
f=\left\{\begin{array} {r@{\;}l}
\hbox{constant}\quad & \hbox{when}\quad \beta<1\\
\hbox{linear\ map}\quad & \hbox{when}\quad \beta=1.
\end{array}
\right.
\]
In particular, $t\mapsto \mathsf{L}_{\alpha,\beta}(f,t)$ is
Lipschitz with respect to $\log\nu_\alpha(t)$ for $t\in (0,1)$.
\end{theorem}

\begin{proof} Similar to that for Theorem \ref{th1}, but this time
by Lemma \ref{l1} (ii).
\end{proof}

Naturally, we can establish the so-called ``norm" estimate
associated to $\Phi_{L,\beta}(f,\cdot)$.

\begin{corollary}\label{co2} Let $0\le\beta\le 1$ and $f\in U(\mathbb D)$ or $f(z)=a_0+a_nz^n$ with $n\in\mathbb N$.
\item{\rm(i)} If $-\infty<\alpha\le -1$, then

$$
\int_0^1 \Phi_{L,\beta}(f,t)\,d\mu_\alpha(t)=\sup_{r\in
(0,1)}\int_0^r \Phi_{L,\beta}(f,t)\,d\mu_\alpha(t)<\infty
$$
if and only if $f$ is constant. Moreover, $\sup_{r\in
(0,1)}\mathsf{L}_{\alpha,\beta}(f,r)=\Phi_{L,\beta}(f,1).$

\item{\rm(ii)} If $-1<\alpha<\infty$, then
$$
\mathsf{L}_{\alpha,\beta}(f,r)\le\mathsf{L}_{\alpha,\beta}(f,1):=\sup_{s\in
(0,1)}\mathsf{L}_{\alpha,\beta}(f,s)\quad\forall\quad r\in (0,1),
$$
where the inequality becomes an equality for all $r\in (0,1)$ if and
only if
\[
f=\left\{\begin{array} {r@{\;}l}
\hbox{constant}\quad & \hbox{when}\quad \beta<1\\
\hbox{linear\ map}\quad & \hbox{when}\quad \beta=1.
\end{array}
\right.
\]

\item{\rm(iii)} $\alpha\mapsto\mathsf{L}_{\alpha,\beta}(f,1)$
is strictly decreasing on $(-1,\infty)$ unless
\[
f=\left\{\begin{array} {r@{\;}l}
\hbox{constant}\quad & \hbox{when}\quad \beta<1\\
\hbox{linear\ map}\quad & \hbox{when}\quad \beta=1.
\end{array}
\right.
\]
\end{corollary}

\begin{proof} The argument is similar to that for Corollary \ref{pr2},
but via Lemma \ref{l1} (ii).
\end{proof}

\section{logarithmic convexity}

In this section, we treat the convexities of the functions: $\log
r\mapsto \log\mathsf{A}_{\alpha,\beta}(f,r)$ and $\log r\mapsto
\log\mathsf{L}_{\alpha,\beta}(f,r)$ for $r\in (0,1)$.

\subsection{Two More Lemmas} The following are two technical preliminaries.

\begin{lemma}\cite[Corollaries 2-3 \& Proposition 7]{WZ}\label{wz}
Suppose $f(x)$ and $\{h_k(x)\}_{k=0}^\infty$ are positive and twice
differentiable for $x\in (0,1)$ such that the function
$H(x)=\sum_{k=0}^\infty h_k(x)$
 is also twice differentiable for $x\in (0,1)$. Then:

\item{\rm(i)} $\log x\mapsto\log f(x)$ is convex if and only if $\log x\mapsto\log f(x^2)$ is convex.
\item{\rm(ii)} The function $\log x\mapsto \log f(x)$ is convex if and only
if the $D$-notation of $f$
$$
D(f(x)):=\frac{f'(x)}{f(x)}+ x\left(\frac{f'(x)}{f(x)}\right)'\ge
0\quad\forall\quad x\in (0,1).
$$
\item{\rm(iii)} If for each $k$ the function $\log x\mapsto \log h_k(x)$ is convex, then $\log x\mapsto \log H(x)$ is also convex.
\end{lemma}

\begin{lemma}\label{uni} Let $f\in H(\mathbb D)$. Then $f$ belongs to $U(\mathbb D)$
provided that one of the following two conditions is valid:

\item{\rm(i)} \cite{Nu} or \cite[Lemma 2.1]{AlD}
$$
f(0)=f'(0)-1=0\quad\&\quad
\left|\frac{z^2f'(z)}{f^2(z)}-1\right|<1\quad\forall\quad z\in
\mathbb D.
$$

\item{\rm(ii)} \cite[Theorem 1]{Ne} or \cite[Theorem 8.12]{Du}
$$
\left|\left[\frac{f''(z)}{f'(z)}\right]'-\frac{1}{2}\left[\frac{f''(z)}{f'(z)}\right]^2\right|\leq
2(1-|z|^2)^{-2}\quad\forall\quad z\in \mathbb D.
$$
\end{lemma}

\subsection{Log-convexity for $\mathsf{A}_{\alpha,\beta}(f,\cdot)$}
Such a property is given below.

\begin{theorem}\label{th3} Let $0\le\beta\le 1$ and $0<r<1$.

\item{\rm(i)} If $\alpha\in (-\infty,-3)$, then there exist $f, g\in H(\mathbb D)$ such that $\log r\mapsto\log\mathsf{A}_{\alpha,\beta}(f,r)$ is not
convex and $\log r\mapsto\log \mathsf{A}_{\alpha,\beta}(g,r)$ is not
concave.
\item{\rm(ii)} If $\alpha\in [-3,0]$, then $\log r\mapsto \log\mathsf{A}_{\alpha,1}(a_nz^n,r\big)$ is
convex for $a_n\not=0$ with $n\in\mathbb N$. Consequently,
$$
\log r\mapsto \log\mathsf{A}_{\alpha,1}\big(f,r\big)
$$
is convex for all $f\in U(\mathbb D)$.

\item{\rm(iii)} If $\alpha\in (0,\infty)$, then $\log r\mapsto\log\mathsf{A}_{\alpha,\beta}(a_nz^n,r)$ is
not convex for $a_n\not=0$ and $n\in \mathbb N$.
\end{theorem}

\begin{proof} The key issue is to check whether or not $\log r\mapsto
\log\mathsf{A}_{\alpha,\beta}(z^n,r)$ is convex for $r\in
(0,1)$.

To see this, let us borrow some symbols from \cite{WZ}. For
$\lambda\ge 0$ and $0<x<1$ we define
$$
f_\lambda (x)=\int_0^x t^\lambda(1-t)^\alpha dt
$$
and
$$
\Delta (\lambda, x)=\frac{f_\lambda'(x)}{f_\lambda
(x)}+x\left(\frac{f_\lambda'(x)}{f_\lambda(x)}\right)'-\left[\frac{f_0'(x)}{f_0(x)}+x\left(\frac{f_0'(x)}{f_0(x)}\right)'\right].
$$

Given $n\in\mathbb N$. A simple calculation shows
$\Phi_{A,\beta}(z^n,t)=\pi^{1-\beta} t^{2(n-\beta)}$, and then a
change of variable derives
\begin{eqnarray*}
\mathsf{A}_{\alpha,\beta}(z^n,r)&=&\frac{\int_0^r
\Phi_{A,\beta}(z^n,t)\,d\mu_\alpha(t)}{\nu_\alpha(r)}\\
&=&\frac{\pi^{1-\beta}\int_0^{r^2}t^{n-\beta}(1-t)^\alpha \,dt
}{\int_0^{r^2} (1-t)^\alpha\, dt}\\
&=& \pi^{1-\beta}\left[\frac{f_{n-\beta}(r^2)}{f_{0}(r^2)}\right].
\end{eqnarray*}

In accordance with Lemma \ref{wz} (i)-(ii), it is readily to work
out that $\log r\mapsto\log\mathsf{A}_{\alpha,\beta}(z^n,r)$ is
convex for $r\in (0,1)$ if and only if $\Delta (n-\beta, x)\ge 0$
for any $x\in (0,1)$.

(i) Under $\alpha\in (-\infty,-3)$, we follow the argument for
\cite[Proposition 6]{WZ} to get
$$
\lim_{x\to 1}\Delta(\lambda,x)=\frac{\lambda (\alpha+1)(\lambda+2+\alpha)}{(\alpha+2)^2(\alpha+3)}.
$$
Choosing
\[
f(z)=z^n=\left\{\begin{array} {r@{\;}l}
z\quad & \hbox{when}\quad \beta<1\\
z^2\quad & \hbox{when}\quad \beta=1
\end{array}
\right.
\]
and $\lambda=n-\beta$, we find $\lim_{x\to 1}\Delta(\lambda,x)<0$,
whence deriving that $\log r\mapsto \log A_\alpha(f,r)$ is not
convex.

In the meantime, picking $n\in \mathbb N$ such that
$n>\beta-(2+\alpha)$ and putting $g(z)=z^n$, we obtain
$$
\lim_{x\to
1}\Delta(n-\beta,x)=\frac{(n-\beta)(\alpha+1)(n-\beta+2+\alpha)}{(\alpha+2)^2(\alpha+3)}>0,
$$
whence deriving that $\log r\mapsto
\log\mathsf{A}_{\alpha,\beta}(g,r)$ is not concave.

(ii) Under $\alpha\in [-3,0]$, we handle the two situations.

{\it Situation 1}: $f\in U(\mathbb D)$. Upon writing
$f(z)=\sum_{n=0}^\infty a_n z^n$, we compute
$$
\Phi_{A,1}\big(f(z),t\big)=(\pi t^2)^{-1}A(f,t)=\sum_{n=0}^\infty
n|a_n|^2 t^{2(n-1)},
$$
and consequently,
$$
\mathsf{A}_{\alpha,1}(f,r)=\frac{\sum_{n=0}^\infty n|a_n|^2\int_0^r
(\pi t^2)^{-1}A(z^n,t)\,
d\mu_\alpha(t)}{\nu_\alpha(r)}=\sum_{n=0}^\infty n |a_n|^2
\mathsf{A}_{\alpha,1}(z^n,r).
$$
So, by Lemma \ref{wz} (iii), we see that the convexity of
$$
\log r\mapsto\log\mathsf{A}_{\alpha,1}(f,r)\quad\hbox{under}\quad
f\in U(\mathbb D)
$$
follows from the convexity of
$$
\log r\mapsto\log\mathsf{A}_{\alpha,1}(z^n,r)\quad\hbox{under}\quad
n\in\mathbb N.
$$
So, it remains to verify this last convexity via the coming-up-next
consideration.

{\it Situation 2}: $f(z)=a_nz^n$ with $a_n\not=0$. Three cases are
required to control.

{\it Case 1}: $\alpha=0$. An easy computation shows
$$
\mathsf{A}_{0,1}(z^n,r)=n^{-1}{r^{2(n-1)}}
$$
and so $\log r\mapsto\log\mathsf{A}_{0,1}(z^n,r)$ is convex.

{\it Case 2}: $-2\le\alpha<0$. Under this condition, we see from the
arguments for \cite[Propositions 4-5]{WZ} that
$$
\Delta(n-1,x)\geq 0\quad\forall\quad n-1\geq 0\ \ \&\ \ 0<x<1,
$$
and so that $\log r\mapsto\log\mathsf{A}_{\alpha,1}(z^n,r)$ is
convex.

{\it Case 3}: $-3\leq \alpha<-2$. With the assumption, we also get
from the arguments for \cite[Propositions 4-5]{WZ} that
$$
\Delta (n-1,x)\geq \Delta(-2-\alpha,x)>0\quad\forall\quad x\in
(0,1)\ \ \&\ \ n-1\in [-2-\alpha,\infty)
$$
and so that $\log r\mapsto\log\mathsf{A}_{\alpha,1}(z^n,r)$ is
convex when $n\ge 2$. Here it is worth noting that the convexity of
$\log r\mapsto\log\mathsf{A}_{\alpha,1}(z,r)=0$ is trivial.

(iii) Under $0<\alpha<\infty$, from the argument for
\cite[Proposition 6]{WZ} we know that $\Delta(n-\beta,x)<0$ as $x$
is sufficiently close to $1$. Thus $\log r\mapsto
\log\mathsf{A}_{\alpha,\beta}(a_n z^n,r)$ is not convex under
$a_n\not=0$.
\end{proof}

The following illustrates that the function $\log r\mapsto
\log\mathsf{A}_{\alpha,\beta}(f,r)$ is not always concave for
$\alpha>0$, $0\le\beta\le 1$, and $f\in U(\mathbb D)$.

\begin{example} Let $\alpha=1$, $\beta\in\{0,1\}$, and $f(z)=z+\frac{z^2}{2}$. Then the function $\log r\mapsto \log\mathsf{A}_{\alpha,\beta}(f,r)$ is neither convex nor
concave for $r\in (0,1)$.
\end{example}
\begin{proof} A direct computation shows
$$
\left|\frac{z^2f'(z)}{f^2(z)}-1\right|=\left|\frac{z^2(1+z)}{(z+\frac{z^2}{2})^2}-1\right|=\frac{|z|^2}{|z+2|^2}<1
$$
since
$$
|z|<1<2-|z|\leq|z+2|\quad\forall\quad z\in \mathbb D.
$$
So, $f\in U(\mathbb D)$ owing to Lemma \ref{uni} (i). By $f'(z)=z+1$
we have
$$
A(f,t)=\int_{t\mathbb D}|z+1|^2\, dA(z)=\pi
\Big(t^2+\frac{t^4}{2}\Big),
$$
plus
\[
\int_0^r \Phi_{A,\beta}(f,t)\,d\mu_1(t)=\left\{\begin{array}
{r@{\;}l}
\frac{\pi}{2}\Big(r^4-\frac{r^6}{3}-\frac{r^8}{4}\Big)\quad & \hbox{when}\quad \beta=0\\
r^2-\frac{r^4}{4}-\frac{r^6}{6}\quad & \hbox{when}\quad \beta=1
\end{array}
\right.
\]
Meanwhile,
$$
\nu_1(r)=\int_0^r (1-t^2)dt^2=r^2-\frac{r^4}{2}.
$$
So, we get
\[
\mathsf{A}_{1,\beta}(f,r)=\left\{\begin{array} {r@{\;}l}
\frac{\pi(12r^2-4r^4-3r^6)}{12(2-r^2)}
\quad & \hbox{when}\quad \beta=0\\
\frac{12-3r^2-2r^4}{6(2-r^2)}\quad & \hbox{when}\quad \beta=1
\end{array}
\right.
\]
and in turn consider the logarithmic convexities of the following
function
\[
h_\beta(x)=\left\{\begin{array} {r@{\;}l}
\frac{12x-4x^2-3x^3}{2-x}\quad & \hbox{when}\quad \beta=0\\
\frac{12-3x-2x^2}{2-x}\quad & \hbox{when}\quad \beta=1
\end{array}
\right.
\]
for $x\in (0,1)$.

Using the so-called D-notation in Lemma \ref{wz}, we have
\[
D(h_\beta(x))=\left\{\begin{array} {r@{\;}l}
D(12x-4x^2-3x^3)-D(2-x)\quad & \hbox{when}\quad \beta=0\\
D(12-3x-2x^2)-D(2-x)\quad & \hbox{when}\quad \beta=1
\end{array}
\right.
\]
for $x\in (0,1)$. By an elementary calculation, we get
\[
\left\{\begin{array} {r@{\;}l}
D(12x-4x^2-3x^3)=\frac{-48-144x+12x^2}{(12-4x-3x^2)^2}\\
D(2-x)=\frac{-2}{(2-x)^2}\\
D(12-3x-2x^2)=\frac{-36-96x+6x^2}{(12-3x-2x^2)^2}.
\end{array}
\right.
\]
Consequently,
\[
D(h_\beta(x))=\left\{\begin{array} {r@{\;}l}
\frac{2g_\beta(x)}{(12-4x-3x^2)^2(2-x)^2}\quad & \hbox{when}\quad \beta=0\\
\frac{2g_\beta(x)}{(12-3x-2x^2)^2(2-x)^2}\quad & \hbox{when}\quad
\beta=1,
\end{array}
\right.
\]
where
\[
g_\beta(x)=\left\{\begin{array} {r@{\;}l}
48-288x+232x^2-72x^3+15x^4\quad & \hbox{when}\quad \beta=0\\
72-192x+147x^2-48x^3+7x^4\quad & \hbox{when}\quad \beta=1.
\end{array}
\right.
\]

Now, under $x\in (0,1)$ we find
$$
g_0'(x)=-288+464x-216x^2+60x^3\quad \&\quad
g_0''(x)=464-432x+180x^2.
$$
Clearly, $g_0''(x)$ is an open-upward parabola with the axis of
symmetry $x=\frac{6}{5}>1$. By $g_0''(1)=212>0$ and the monotonicity
of $g_0''$ on $(0,1)$, we have $g_0''(x)>0$ for all $x\in (0,1)$.
Thus $g_0'$ is increasing on $(0,1)$. The following condition
$$
g_0'(0)=-288<0\quad \&\quad  g_0'(1)=20>0
$$
yields an $x_1\in (0,1)$ such that $g_0'(x)<0$ for $x\in(0,x_1)$ and
$g_0'(x)>0$ for $x\in (x_1,1)$. Since $g_0(0)=48$ and $g_0(1)=-65$,
there exists an $x_0\in (0,1)$ such that $g_0(x)>0$ for $x\in
(0,x_0)$ and $g_0(x)<0$ for $x\in (x_0,1)$. Thus the function $\log
x\mapsto\log h_0(x)$ is neither convex nor concave.

Similarly, under $x\in (0,1)$ we have
$$
g_1'(x)=-192+294x-144x^2+28x^3\quad \&\quad g_1''(x)=294-288x+84x^2.
$$
Obviously, $g_1''(x)$ is an open-upward parabola with the axis of
symmetry $x=\frac{12}{7}>1$. By $g_1''(1)=90>0$ and the monotonicity
of $g_1''$ on $(0,1)$, we have $g_1''(x)>0$ for all $x\in (0,1)$.
Thus $g_1'$ is increasing on $(0,1)$. The following condition
$$
g_1'(0)=-192<0\quad \&\quad  g_1'(1)=-14<0
$$
yields  $g_1'(x)<0$ for $x\in(0,1)$. Since $g_1(0)=72$ and
$g_1(1)=-14$, there exists an $x_0\in (0,1)$ such that $g_1(x)>0$
for $x\in (0,x_0)$ and $g_1(x)<0$ for $x\in (x_0,1)$. Thus the
function $\log x\mapsto\log h_1(x)$ is neither convex nor concave.
\end{proof}

\subsection{Log-convexity for $\mathsf{L}_{\alpha,\beta}(f,\cdot)$} Analogously, we can
establish the expected convexity for the mixed lengths.

\begin{theorem}\label{th4} Let $0\le\beta\le 1$ and $0<r<1$.

\item{\rm(i)} If $\alpha\in (-\infty,-3)$, then there exist $f, g\in H(\mathbb D)$ such that $\log r\mapsto\log\mathsf{L}_{\alpha,\beta}(f,r)$ is not
convex and $\log r\mapsto\log \mathsf{L}_{\alpha,\beta}(g,r)$ is not
concave.
\item{\rm(ii)} If $\alpha\in [-3,0]$, then $\log r\mapsto \log\mathsf{L}_{\alpha,1}(a_nz^n,r\big)$ is
convex for $a_n\not=0$ with $n\in\mathbb N$. Consequently, $\log
r\mapsto \log\mathsf{L}_{\alpha,1}(f,r)$ is convex for $f\in
U(\mathbb D)$.

\item{\rm(iii)} If $\alpha\in (0,\infty)$, then $\log r\mapsto\log\mathsf{L}_{\alpha,\beta}(a_nz^n,r)$ is not convex
for $a_n\not=0$ and $n\in \mathbb N$.
\end{theorem}

\begin{proof} The argument is similar to that for Theorem \ref{th3}
except using the following statement for $\alpha\in [-3,0]$ -- If
$f\in U(\mathbb D)$, then there exists $g(z)=\sum_{n=0}^\infty b_n
z^n$ such that $g$ is the square root of the zero-free derivative
$f'$ on $\mathbb D$ and $f'(0)=g^2(0)$, and hence
\begin{eqnarray*}
\Phi_{L,1}(f,t)&=&(2\pi t)^{-1}\int_{t\mathbb
T}|f'(z)||dz|\\
&=&(2\pi t)^{-1} \int_{t\mathbb T} |g(z)|^2|dz|\\
&=&\sum_{n=0}^\infty |b_n|^2 t^{2n}.
\end{eqnarray*}
\end{proof}

Our concluding example shows that under $0<\alpha<\infty$ and
$0\le\beta\le 1$ one cannot get that
$\log\mathsf{L}_{\alpha,\beta}(f,r)$ is convex or concave in $\log
r$ for all functions $f\in U(\mathbb D)$.

\begin{example} Let $\alpha=1$, $\beta\in\{0,1\}$, and $f(z)=(z+2)^3$. Then the function $\log r\mapsto\log\mathsf{L}_{\alpha,\beta}(f,r)$ is neither convex nor concave for $r\in (0,1)$.
\end{example}

\begin{proof} Clearly, we have

$$
f'(z)=3(z+2)^2\ \ \&\ \ f''(z)=6(z+2)
$$
as well as the Schwarizian derivative
$$
\left[\frac{f''(z)}{f'(z)}\right]'-\frac{1}{2}\left[\frac{f''(z)}{f'(z)}\right]^2=\frac{-4}{(z+2)^2}.
$$
It is easy to see that
$$
\sqrt{2}(1-|z|^2)\leq 2-|z|\quad\forall\quad  z\in \mathbb D.
$$
So,
$$
\left|\left[\frac{f''(z)}{f'(z)}\right]'-\frac{1}{2}\left[\frac{f''(z)}{f'(z)}\right]^2\right|=\frac{4}{|z+2|^2}\leq
\frac{4}{(2-|z|)^2}\leq \frac{2}{(1-|z|^2)^2}.
$$
By Lemma \ref{uni} (ii), $f$ belongs to $U(\mathbb D)$.
Consequently,
$$
L(f,t)=\int_0^{2\pi}|f'(te^{i\theta})|t\, d\theta=6\pi t(t^2+4)
$$
and
\[
\int_0^r\Phi_{L,\beta}(f,t)\,d\mu_1(t)=\left\{\begin{array}
{r@{\;}l}
12\pi\Big(\frac{4}{3}r^3-\frac{3}{5}r^5-\frac{1}{7}r^7\Big)
\quad & \hbox{when}\quad \beta=0\\
12r^2-\frac{9}{2}r^4-r^6\quad & \hbox{when}\quad \beta=1.
\end{array}
\right.
\]
Note that $\nu_1(r)=r^2-\frac{r^4}{2}$. So,
\[
\mathsf{L}_{1,\beta}(f,r)=\left\{\begin{array} {r@{\;}l}
\frac{24\pi(140r-63r^3-15r^5)}{105(2-r^2)}
\quad & \hbox{when}\quad \beta=0\\
\frac{24-9r^2-2r^4}{2-r^2} \quad & \hbox{when}\quad \beta=1.
\end{array}
\right.
\]

To gain our conclusion, we only need to consider the logarithmic
convexity of the function
\[
 h_\beta(x)=\left\{\begin{array} {r@{\;}l}
\frac{140x-63x^3-15x^5}{2-x^2}\quad & \hbox{when}\quad \beta=0\\
\frac{24-9x-2x^2}{2-x}\quad & \hbox{when}\quad \beta=1.
\end{array}
\right.
\]

{\it Case 1}: $\beta=0$. Applying the definition of $D$-notation, we
obtain
$$
D(140x-63x^3-15x^5)=\frac{-35280 x-33600
x^3+3780x^5}{(140-63x^2-15x^4)^2}
$$
and
$$
D(2-x^2)=\frac{-8x}{(2-x^2)^2},
$$
whence reaching
$$
D\big(h_0(x)\big)=D(140x-63x^3-15x^5)-D(2-x^2)=\frac{4xg_0(x)}{(140-63x^2-15x^4)^2(2-x^2)^2},
$$
where
$$
g_0(x)=3920-33600x^2+28098x^4-8400x^6+1395x^8.
$$
Obviously,
$$
g_0(0)=3920>0\quad\&\quad g_0(1)=-8587<0.
$$
Now letting $s=x^2$, we get
$$
g_0(x)=G_0(s)=3920-33600s+28098s^2-8400s^3+1395s^4,
$$
and
$$
G'_0(s)=-33600+56196s-25200s^2+5580s^3\ \&\
G''_0(s)=56196-50400s+16740s^2.
$$
Since the axis of symmetry of $G''_0$ is $s=\frac{140}{93}>1$,
$G''_0$ is decreasing on $(0,1)$. Due to $G''_0(1)=22536>0$, we have
$G''_0(s)>0$ for all $s\in (0,1)$, i.e., $G'_0(s)$ is increasing on
$(0,1)$. By
$$
G'_0(0)=-33600<0\quad\&\quad G'_0(1)=2976>0,
$$
we conclude that there exists an $s_0\in(0,1)$ such that $G'_0(s)<0$
for $s\in(0,s_0)$ and $G'_0(s)>0$ for $s\in (s_0,1)$. Then there
exists an $x_0\in (0,1)$ such that $g_0(x)$ is decreasing for $x\in
(0,x_0)$ and $g_0(x)$ is increasing for $x\in(x_0,1)$. Thus there
exists an $x_1\in (0,1)$ such that $g_0(x)>0$ for $x\in(0,x_1)$ and
$g_0(x)<0$ for $x\in(x_1,1)$. As a result, we find that $\log
r\mapsto\log\mathsf{L}_{\alpha,0}(f,r)$ is neither concave nor
convex.

{\it Case 2}: $\beta=1$. Again using the $D$-notation, we obtain
$$
D(24-9x-2x^2)=\frac{-216-192x+18x^2}{(24-9x-2x^2)^2}
$$
and
$$
D(2-x)=\frac{-2}{(2-x)^2},
$$
whence deriving
$$
D\big(h_1(x)\big)=D(24-9x-2x^2)-D(2-x)=\frac{2g_1(x)}{(24-9x-2x^2)^2(2-x)^2},
$$
where
$$
g_1(x)=144-384x+297x^2-96x^3+13x^4.
$$
Now we have
$$
g'_1(x)=-384+594x-288x^2+52x^3\quad \&\quad
g''_1(x)=594-576x+156x^2.
$$
Since the axis of symmetry of $g''_1(x)$ is $x=\frac{24}{13}>1$,
$g''_1(x)$ is decreasing on $(0,1)$. Due to $g''_1(1)=174>0$, we
have $g''_1(x)>0$ for all $x\in (0,1)$, i.e., $g'_1(x)$ is
increasing on $(0,1)$. By
$$
g'_1(0)=-384<0\quad\&\quad g'_1(1)=-26<0,
$$
we conclude that $g'_1(x)<0$ for $x\in (0,1)$. Obviously,
$$
g_1(0)=144>0\quad \&\quad g_1(1)=-26<0.
$$
Hence there exists an $x_0\in (0,1)$ such that $g_1(x)>0$ for $x\in
(0,x_0)$ and $g_1(x)<0$ for $x\in (x_0,1)$. Consequently, we find
that $\log r\mapsto\log\mathsf{L}_{\alpha,\beta=1}(f,r)$ is neither
concave nor convex.
\end{proof}

\end{document}